\documentclass{amsart}
\usepackage{amsmath, amsthm, amsfonts, amssymb}

\usepackage[pdftex]{hyperref}

\hypersetup{ colorlinks=true, linkcolor=blue, filecolor=magenta, urlcolor=cyan, }

\hypersetup{
	colorlinks=true,	
	linkcolor=cyan,		
	citecolor=blue,		
	filecolor=red,		
	urlcolor=magenta	
}

\newcommand{\RR}{\mathbb{R}} 
\newcommand{\CC}{\mathbb{C}} 

\newcommand{\UB}{\mathbb{B}}

\newcommand{\trace}{\mathrm{trace}}

\newcommand{\rank}{\mathrm{rank}\,}

\newcommand{\grad}{\,\mathrm{grad}\,}

\newcommand{\Ric}{\mathrm{Ric}}
\newcommand{\cK}{\mathsf{K}}
\newcommand{\cC}{\mathsf{C}}
\newcommand{\cL}{\mathsf{L}}

\newcommand{\dc}{d^c}

\newcommand{\vV}{\mathcal{V}}

\newcommand{\afrac}[2]{\genfrac{}{}{0pt}{}{#1}{#2}}

\newcommand{\paren}[1]{\left(#1\right)}
\newcommand{\abs}[1]{\left\lvert#1\right\rvert}
\newcommand{\norm}[1]{\left\|#1\right\|}
\newcommand{\set}[1]{\left\{#1\right\}}
\newcommand{\ip}[1]{\left\langle#1\right\rangle}
\newcommand{\lb}[1]{\left[#1\right]}
\newcommand{\pd}[2]{\frac{\partial#1}{\partial#2}}
\newcommand{\spd}[3]{\frac{\partial^2#1}{\partial#2\partial#3}}
\newcommand{\pdl}[2]{\partial#1/\partial#2}

\newcommand{\ind}[4]{{#1}^{\phantom{#2}#3}_{#2\phantom{#3}#4}}

\newcommand{\im}{\sqrt{-1}}

\newcommand{\KE}{K\"ahler-Einstein }

\newtheorem{theorem}{Theorem}[section]
\newtheorem{corollary}[theorem]{Corollary}
\newtheorem{lemma}[theorem]{Lemma}
\newtheorem{proposition}[theorem]{Proposition}
\theoremstyle{definition}

\theoremstyle{remark}
\newtheorem{remark}[theorem]{Remark}

\numberwithin{equation}{section}

\title{On the K\"ahler-hyperbolicity of bounded symmetric domains}

\author{Young-Jun Choi}
\address{Department of Mathematics, 
Pusan National University,
2, Busandaehak-ro 63beon-gil,
Geumjeong-gu, Busan, 46241, Republic of Korea}
\email{youngjun.choi@pusan.ac.kr}

\author{Kang-Hyurk Lee}
\address{Department of Mathematics and Research Institute of Natural Science, 
Gyeongsang National University, 
Jinju, Gyeongnam, 52828, Republic of Korea}
\email{nyawoo@gnu.ac.kr}

\author{Aeryeong Seo}
\address{Department of Mathematics and RIRCM,
Kyungpook National University,
80, Daehak-ro, Buk-gu, Daegu, 41566, Republic of Korea}
\email{aeryeong.seo@knu.ac.kr}



\thanks{This work was supported by Samsung Science and Technology Foundation under Project Number SSTF-BA2201-01. The first named author was supported by the National Research Foundation of Korea (NRF) grant funded by the Korea government (No.~NRF-2023R1A2C1007227). The third  named author was supported by the National Research Foundation of Korea (NRF) grant funded by the Korea government (No.~NRF-2022R1F1A1063038). }

\begin{document}

\begin{abstract}
In this paper, we characterize the K\"ahler-hyperbolicity length of a bounded symmetric domain, defined by its rank and genus, as a unique constant determined by a constant gradient length of a special Bergman potential. Additionally, we establish a characterization of the lower bound of $L^\infty$ norm of the gradient length of any Bergman potential.
\end{abstract}


\maketitle

\section{Introduction}


Let $X$ be an $n$-dimensional complete K\"ahler manifold and $\omega$ be its K\"ahler form. We say that $\omega$ is \emph{$d$-bounded} if there exists a real $1$-form $\eta$ on $X$ such that 
\begin{equation*}
d\eta=\omega\quad\text{and}\quad \norm{\eta}_{L^\infty} = \sup_X \abs{\eta}_\omega<+\infty
\end{equation*}
where $\abs{\eta}_\omega$ denotes the pointwise length of $\eta$ measured by $\omega$. In \cite{Gromov1991}, Gromov proved that if X is simply connected and admits a $d$-bounded K\"ahler form $\omega$, then 
\begin{equation*}
\dim H^{p,q}_2 (X) = 0 \quad (p+q\neq n)
\end{equation*}
where $H^{p,q}_2(X)$ is the space  of square integrable harmonic $(p,q)$-forms with respect to $\omega$. Moreover, he also proved that if $X$ covers a compact K\"ahler manifold and $\omega$ is a lifting of a K\"ahler form of the quotient, then \begin{equation*}
\dim H^{p,q}_2 (X) = \infty \quad (p+q=n).
\end{equation*}
Subsequently, Donnelly generalized the vanishing result to complete K\"ahler manifolds \cite{Donnelly1994}.

In what follows, we will call a K\"ahler structure to be \emph{K\"ahler hyperbolic} if its K\"ahler  form is $d$-bounded (see \cite{Gromov1991} for the original definition of K\"ahler hyperbolicity).

It is worth mentioning that analogous results in the context of cohomology vanishing theorems were obtained earlier by Donnelly-Fefferman~\cite{Donnelly-Fefferman1983}. They studied $L^2$ cohomology groups of bounded strongly pseudoconvex domains in the complex Euclidean space $\mathbb C^n$ with respect to their Bergman metrics, employing deep analytic tools from several complex variables.
Given these results, it is natural to study the $d$-boundedness of the Bergman metric on pseudoconvex domains in $\mathbb C^n$ when the metric is complete.

For a domain $\Omega$ in $\mathbb C^n$, the Bergman metric is defined by $\omega=dd^c\log K_\Omega$, where $K_\Omega=K_\Omega(z,z)$ is the Bergman kernel function  of $\Omega$. When $\Omega$ is pseudoconvex, it has been studied to establish the boundedness of $\abs{d^c\log K_\Omega}_\omega$, equivalently that of $\abs{\partial\log K_\Omega}_\omega$,  in order to obtain the K\"ahler-hyperbolicity of the Bergman metric (\cite{Donnelly1994,Donnelly1997,McNeal2002,Chen2004,Yeung2006}).

For a bounded homogeneous domain $\Omega$,  Donnelly  proved in \cite{Donnelly1997} the uniform boundedness of  $\abs{\partial\log K_\Omega}_\omega$, the \emph{gradient length} of $\log K_\Omega$, using Gindikin's analysis (\cite{Gindikin1964}) on holomorphic automorphism groups. 
Later, in \cite{Kai-Ohsawa2007}, Kai-Ohsawa obtained the same result by a simple and elementary approach, taking advantage of its realization as a Siegel domain. Their idea can be outlined as follows: through the Cayley transform, each $\Omega$ is biholomorphic to a homogeneous Siegel domain $D$ of the second kind. Since $D$ is affine homogeneous (\cite{KMO1970}) and the Bergman metrics are invariant under biholomorphisms, the K\"ahler potential $\log K_D$ of the Bergman metric on $D$ has then constant gradient length. Pulling back the potential $\log K_D$ of $D$ by the Cayley transform to $\Omega$ gives a K\"ahler potential of the Bergman metric on $\Omega$ with constant gradient length.

Interestingly, in \cite{Kai-Ohsawa2007}, Kai and Ohsawa showed that if a K\"ahler potential $\varphi$ of the Bergman metric on a bounded homogeneous domain has constant gradient length, then this length is uniquely determined, that is, it is the same as $\abs{\partial\log K_D}_\omega$.
They also posed the following natural question:
\begin{center}
\emph{How does the length of $\partial\varphi$ depend on $\Omega$?}
\end{center}

\medskip

The aim of this paper is to generalize their result for local K\"ahler potential of the Bergman metric and to provide an answer to this question when $\Omega$ is a bounded symmetric domain of higher rank, possibly in the reducible case, and to explicitly describe its constant length (Theorem~\ref{thm:main1}). Note that for the unit ball, this question was already resolved by Lee-Seo~\cite[Theorem~3.3]{LS2024}. We also show that this constant length is the minimal value of $\norm{\partial\varphi}_{L^\infty}$ for a K\"ahler potential $\varphi$ of the Bergman metric on bounded symmetric domains (Theorem~\ref{thm:main2}).

\subsection{The K\"ahler-hyperbolic length and the rigidity of constant gradient length}
Let $\Omega$ be an irreducible bounded symmetric domain in $\CC^n$, especially a Cartan/Harish-Chandra embedding of the corresponding Hermitian symmetric space of noncompact type, and let $N_\Omega$ be its generic norm. The Bergman kernel $K_\Omega$ of $\Omega$ is  then of the form 
\begin{equation*}
K_\Omega(z,w) 
	= c N_\Omega(z,w)^{-c_\Omega}
\end{equation*}
for a normalizing constant $c$ by the Euclidean volume of $\Omega$ and the genus $c_\Omega$ of $\Omega$ which is a positive integer. Then the Bergman metric $\omega=dd^c\log K$ is the complete \KE metric with Ricci curvature $-1$, that is, $\Ric(\omega)=-\omega$.

Now we define the \emph{K\"ahler-hyperbolicity length} $\cL_\Omega$ of $(\Omega,\omega)$ by
\begin{equation*}
\cL_\Omega = \sqrt{rc_\Omega}
\end{equation*}
where $r$ is the rank of $\Omega$, the dimension of a maximal totally geodesic polydisc in $\Omega$. If a bounded symmetric domain $\Omega$ is biholomorphic to a product $\Omega_1\times\cdots\times\Omega_s$ of irreducible bounded symmetric domains, then we set
\begin{equation*}
\cL_\Omega 
	= \paren{\sum_{j=1}^s \cL_{\Omega_j}^2}^{1/2} 
	=\paren{ \sum_{j=1}^s r_j c_{\Omega_j} }^{1/2}
\end{equation*}
where each $r_j$ and $c_{\Omega_j}$ are the rank and the genus of $\Omega_j$.

The first main theorem of this paper answers the question posed in \cite{Kai-Ohsawa2007} in a more general context, specifically for the local K\"ahler potentials of the Bergman metrics.

\begin{theorem}\label{thm:main1}
Let $\Omega$ be a bounded symmetric domain with the Bergman metric $\omega$. If there is a local potential function $\varphi$ of $\omega$ with constant gradient length, then the length is the same as the K\"ahler-hyperbolicity length $\cL_\Omega$ of $(\Omega,\omega)$:
\begin{equation*}
\abs{\partial\varphi}_\omega\equiv\cL_\Omega\;.
\end{equation*}
\end{theorem}
We will prove this theorem by showing that the constant length is indeed determined by the Gaussian curvature of a totally geodesic holomorphic disc in $\Omega$ which is an integrable submanifold tangential to the holomorphic distribution generated by the gradient vector field of the potential function.


\subsection{The optimal value for the $L^\infty$ norm of Bergman potential's gradient length} 
Let $X$ be a universal covering of a compact K\"ahler manifold $Y$ and $\omega$ be a lifting of a K\"ahler form of $Y$. Suppose that $\omega$ is $d$-bounded, i.e. there is  a global $1$-form $\eta$ with $d\eta=\omega$ and $\norm{\eta}_{L^\infty}<+\infty$.  Then in Theorem 1.4.A of \cite{Gromov1991}, Gromov showed that the first eigenvalue $\lambda$ of the Laplace operator $\triangle=d\delta+\delta d$ in the space of square integrable $(p,q)$-forms for $p+q\neq n$ has the following lower bound:
\begin{equation*}
\lambda \geq \frac{c_n}{\norm{\eta}_{L^\infty}^2}
\end{equation*}
where $c_n$ is a constant depending only on the dimension of $X$.

When we consider the bounded symmetric domain $\Omega$, the Bergman metric $\omega$ is a canonical lifting of a \KE metric of a compact quotient of $\Omega$. Therefore it is natural to find the minimal value of $\norm{\eta}_{L^\infty}$ to sharpen the lower bound estimate. Here, we will show that the minimal value is determined by the K\"ahler-hyperbolicity length $\cL_\Omega$ of $(\Omega,\omega)$.
\begin{theorem}\label{thm:main2}
Let $\Omega$ be a bounded symmetric domain and $\omega$ be its Bergman metric. 
For any global potential function $\varphi\colon\Omega\to\RR$, the following inequality holds.
\begin{equation}\label{lower bound}
\norm{\partial\varphi}_{L^\infty} \geq \cL_\Omega \;.
\end{equation}
\end{theorem}
By Theorem~\ref{thm:main1} and the result of Kai-Ohsawa \cite{Kai-Ohsawa2007}, we have a global potential satisfying the equality in \eqref{lower bound}. Consequently, $\cL_\Omega$ is the minimal value for the $L^\infty$ norms of the gradients of K\"ahler potentials of $\omega$.

For $\eta=d^c\varphi$, where $d^c = \frac{\im}{2}(\bar\partial-\partial)$, we have $d\eta = \omega$
and 
\begin{equation*}
\norm{\eta}_{L^\infty} = \frac{1}{\sqrt{2}}\norm{\partial\varphi}_{L^\infty},
\end{equation*}
since $2\abs{\eta}_\omega^2=\abs{\partial\varphi}_\omega^2$. 

Meanwhile, when a $1$-form $\eta$ with $d\eta=\omega$ is given and $\eta$ is $d^c$-closed, we can solve $d^c\varphi=\eta$ since $\Omega$ is contractible. Thus Theorem~\ref{thm:main2} implies that
\begin{corollary}\label{cor:estimate of length}
Let $\eta$ be a global $1$-form on a bounded symmetric domain $\Omega$. If $d\eta$ is the K\"ahler form of the Bergman metric and $d^c\eta=0$, then
\begin{equation*}
\norm{\eta}_{L^\infty}^2 \geq \frac{\cL_\Omega^2}{2}\;.
\end{equation*}
\end{corollary}


\subsection{Organization of proofs using the complete \KE metric}
We will prove the main theorems by reducing the problems to certain totally geodesic complex submanifolds of the given bounded symmetric domain. The restriction of the metric to such submanifold is not its Bergman metric, but still its complete \KE metric. For transparent statements, we will deal with the complete \KE metric $\omega$ of $\Omega$ with normalized Ricci curvature $-\cK$.

Since the Bergman metric of bounded homogeneous domain is \KE with Ricci curvature $-1$, the metric given by
\begin{equation*}
\omega = \frac{1}{\cK}dd^c\log K_\Omega
\end{equation*}
is the unique complete \KE metric of $\Omega$ with Ricci curvature $-\cK$ by Yau's Schwarz lemma (\cite{Yau1978-1}). 
Therefore, to prove Theorem~\ref{thm:main1} and Theorem~\ref{thm:main2}, it suffices to establish the following equivalent theorem.
\begin{theorem}\label{thm:main KE}
Let $\omega$ be the complete \KE metric of a bounded symmetric domain $\Omega$ with Ricci curvature $-\cK$: $\Ric(\omega)=-\cK\omega$.
\begin{enumerate}
\item\label{thm:main KE1} If there is a local potential $\varphi$ of $\omega$ with constant gradient length, then 
\begin{equation*}
\abs{\partial\varphi}_\omega^2 \equiv \frac{\cL_\Omega^2}{\cK} \;.
\end{equation*}
\item\label{thm:main KE2} For any global potential $\varphi$ of $\omega$, we have 
\begin{equation*}
\sup_\Omega\abs{\partial\varphi}_\omega^2 \geq \frac{\cL_\Omega^2}{\cK} \;.
\end{equation*}
\end{enumerate}
\end{theorem}

The organization of this paper is as follows.
In \S\ref{sec:gradient}, we study the integral submanifold of the distribution generated by the gradient vector field of the K\"ahler potential. 
Moreover, we will reduce the proof of main theorem to this integral submanifold (Theorem~\ref{thm:integral curve of gradient vector field}).
In \S\ref{sec:polydisc}, we present the geometry of polydisc $\Delta^n$.
In particular, we determine the constant gradient length of a local potential of the Poincar\`e metric on $\Delta^n$ (Proposition~\ref{prop:polydisc uniqueness}). To achieve this, we prove Lemma~\ref{lem:rank} which also plays a crucial role to establish the main theorems.
Remark~\ref{rmk:reducible} discusses the analogous results presented in \S\ref{sec:polydisc} for reducible bounded symmetric domains case.
In \S\ref{sec:proof}, we describe the Gaussian curvature of the totally geodesic holomorphic disc in a maximal polydisc of the bounded symmetric domain in terms of its rank and genus. By using this description, we prove the main theorems in \S\ref{subsec:proof1} and \S\ref{subsec:proof2}.


\section{The gradient vector field of a K\"ahler potential}\label{sec:gradient}




In this section, we introduce the basic notions of the K\"ahler metric and examine the properties of the gradient vector field associated with a K\"ahler potential. In particular, we will show that if a K\"ahler manifold admits a potential with a gradient of constant length, this length is determined by the Gaussian curvature of a totally geodesic holomorphic curve.

\begin{theorem}\label{thm:integral curve of gradient vector field}
If a K\"ahler manifold $(X,\omega)$ admits a local potential $\varphi$ of $\omega$ satisfying
\begin{equation*}
\abs{\partial\varphi}_\omega^2\equiv \cC \;
\end{equation*}
for some constant $\cC$,
then the distribution generated by the gradient vector field $\vV$ of $\varphi$ is involutive. Moreover, each integral submanifold of this distribution is a totally geodesic holomorphic curve with constant Gaussian curvature:
\begin{equation*}
\kappa\equiv-\frac{2}{\cC}\;.
\end{equation*}
\end{theorem}

As we will define in \eqref{eqn:def of gradient}, the gradient vector field of $\varphi$ in this paper is the $(1,0)$-component of the Riemannian gradient vector field, so it is a section to the $(1,0)$-tangent bundle $T^{(1,0)}X$. Thus the distribution generated by $\vV$ means the rank $2$ distribution of the real tangent bundle $TX$ generated by $\vV+\overline\vV$ and $\im\paren{\vV-\overline\vV}$.

\medskip

In the Riemann surface case with a hermitian metric, the existence of a potential with constant gradient length is equivalent to the Gaussian curvature being a negative constant.

\begin{corollary}
Let $S$ be a Riemann surface with a hermitian metric $ds^2$. If there is a potential function $\varphi$ with $\abs{\partial\varphi}_{ds^2}^2\equiv \cC$, then the Gaussian curvature of $ds^2$ is constant: $\kappa\equiv -2/\cC$.
\end{corollary}
Therefore we have the $1$-dimensional version of Theorem~\ref{thm:main1} and Statement (\ref{thm:main KE1}) of Theorem~\ref{thm:main KE}.

\begin{corollary}
Let $\omega_\Delta$ be the Poincar\'e metric of the unit disc $\Delta=\set{\zeta\in\CC: \abs{\zeta}<1}$ with Gaussian curvature  $-\cK$. If there is a local potential function $\varphi$ of $\omega_\Delta$ with constant gradient length, then 
\begin{equation*}
\abs{\partial\varphi}_{\omega_\Delta}^2\equiv\frac{2}{\cK} \;.
\end{equation*}
\end{corollary}
Note that the rank of the unit disc is $1$ and the genus $c_\Delta$ is $2$. Thus the K\"ahler-hyperbolicity length of $\Delta$ is given by $\cL_\Delta^2=2$.


\subsection{Basic notations of K\"ahler manifolds}
Let $X$ be an $n$-dimensional complex manifold with a K\"ahler metric $\omega$.  In a local holomorphic coordinate system $z=(z^1,\ldots,z^n)$, the K\"ahler form $\omega$ can be written as
\begin{equation*}
\omega = \im g_{\alpha\bar\beta} dz^\alpha\wedge dz^{\bar\beta}
\end{equation*}
where the metric tensor $g_{\alpha\bar\beta}$ stands for the value of $\ip{\pdl{}{z^\alpha},\pdl{}{z^{\bar\beta}}}_\omega$. Throughout this paper, the indices $\alpha,\beta,\ldots$ run from $1$ to $n$ and the summation convention for duplicated indices is always assumed. We denote the complex conjugate of a tensor by taking the bar on the indices: $\overline{z^\alpha} = z^{\bar\alpha}$, $\overline{g_{\alpha\bar\beta}} = g_{\bar\alpha\beta}$. Since $(g_{\alpha\bar\beta})$ is hermitian, we have $g_{\alpha\bar\beta}=g_{\bar\beta\alpha}$.

In this coordinate expression, we have the inverse matrix $(g^{\bar\beta\alpha})$ of $(g_{\alpha\bar\beta})$: $g^{\alpha\bar\gamma}g_{\bar\gamma\beta}=\ind{\delta}{}{\alpha}{\beta}$. We also use  $(g_{\alpha\bar\beta})$ and  $(g^{\bar\beta\alpha})$ to raise and lower indices in the tensor calculus: $
\varphi^\alpha = g^{\alpha\bar\beta}\varphi_{\bar\beta}$, $R_{\alpha\bar\beta\lambda\bar\mu} = g_{\bar\beta\gamma}\ind{R}{\alpha}{\gamma}{\lambda\bar\mu}$. 
Here, $R_{\alpha\bar\beta\lambda\bar\mu}$ and $\ind{R}{\alpha}{\gamma}{\lambda\bar\mu}$ represent the coefficient of the Riemannian curvature tensor
\begin{equation*}
R(\partial_\lambda,\partial_{\bar\mu})\partial_\alpha 
	= \nabla_{\partial_\lambda}\nabla_{\partial_{\bar\mu}}\partial_\alpha
	-\nabla_{\partial_{\bar\mu}}\nabla_{\partial_\lambda}\partial_\alpha
	=\ind{R}{\alpha}{\gamma}{\lambda\bar\mu}\partial_\gamma
\end{equation*}
where $\partial_\alpha=\pdl{}{z^\alpha}$, $\partial_{\bar\beta}=\pdl{}{z^{\bar\beta}}$ and $\nabla$ is the K\"ahler connection of $\omega$. Then the Ricci curvature tensor can be written as $\Ric _\omega = R_{\alpha\bar\beta} dz^\alpha\otimes dz^{\bar\beta}$ for $R_{\alpha\bar\beta}=g^{\lambda\bar\mu} R_{\alpha\bar\beta\lambda\bar\mu}$. The Ricci form $\Ric(\omega)$ of $\omega$,  an intrinsically defined global $(1,1)$-form, is given by
\begin{equation*}
\Ric(\omega)= \im R_{\alpha\bar\beta} dz^\alpha\wedge dz^{\bar\beta}
\end{equation*}
which can be obtained by
\begin{equation*}
\Ric(\omega) = -dd^c \log\det(g_{\alpha\bar\beta}) 
\end{equation*}
where $\dc=\frac{\im}{2}(\bar\partial-\partial)$. 
Note that $d\dc=\im\partial\bar\partial$.

\medskip

For a K\"ahler manifold $(X,\omega)$, a \emph{local potential} of $\omega$ is a local smooth function $\varphi$  satisfying
\begin{equation*}
d\dc\varphi = \omega \;.
\end{equation*}
This is equivalent to $\varphi_{\alpha\bar\beta}=g_{\alpha\bar\beta}$ for any $\alpha,\beta$ under the notations  $\varphi_\alpha = \pdl{\varphi}{z^\alpha}$, $\varphi_{\bar\beta}=\pdl{\varphi}{z^{\bar\beta}}$ and $\varphi_{\alpha\bar\beta} = \partial^2\varphi/\partial z^\alpha\partial z^{\bar\beta}$.

The \emph{gradient length} of a potential $\varphi$ is the length of the $(1,0)$-form $\partial\varphi$ measured by $\omega$ which is locally expressed by
\begin{equation*}
\abs{\partial\varphi}_\omega^2 
	=\ip{\partial\varphi,\bar\partial\varphi}_\omega
	=\ip{\varphi_\alpha dz^\alpha, \varphi_{\bar\beta} dz^{\bar\beta}}_\omega
	= \varphi_\alpha\varphi_{\bar\beta}g^{\alpha\bar\beta} =\varphi_\alpha\varphi^\alpha
	\;.
\end{equation*}
This quantity coincides with the metric length of the \emph{gradient vector field} of $\varphi$,
\begin{equation*}
\vV=\grad(\varphi) 
	= \varphi^\alpha\partial_\alpha,
\end{equation*}
which is uniquely defined by the $(1,0)$-tangent vector field $\vV$ satisfying 
\begin{equation}\label{eqn:def of gradient}
\partial\varphi(v) = \ip{v,\overline\vV}_\omega \quad\text{for any $(1,0)$-tangent vector $v$ of $X$. }
\end{equation}


\subsection{A potential with constant gradient length}
In order to prove Theorem~\ref{thm:integral curve of gradient vector field}, we need the following basic properties of the gradient vector field of a K\"ahler potential.
\begin{proposition}\label{prop:on gradient vector field}
Let $(X,\omega)$ be a K\"ahler manifold and $\varphi$ be a potential of $\omega$. For the gradient vector field $\vV=\grad(\varphi)$, the followings hold.
\begin{enumerate}
\item For any $(1,0)$-tangent vector $v$, 
\begin{equation*}
\nabla_v \vV = v  \;.
\end{equation*}
\item If $\abs{\partial\varphi}_\omega$ is constant, then
\begin{equation*}
\nabla_{\overline\vV}\vV = -\vV  \;.
\end{equation*}
\end{enumerate}
\end{proposition}

\begin{proof}
Let us denote the covariant derivative of $\vV$ by
\begin{equation*}
\nabla  \vV = \nabla \paren{\varphi^\alpha\partial_\alpha} 
	= \paren{
		\ind{\varphi}{}{\alpha}{;\beta} dz^\beta
		+\ind{\varphi}{}{\alpha}{;\bar\beta} dz^{\bar\beta}
		}\otimes\partial_\alpha\;.
\end{equation*}
using the semicolon notation. Since $\varphi_{\bar\gamma;\beta} =\varphi_{\bar\gamma\beta}=g_{\bar\gamma\beta}$, we have $\ind{\varphi}{}{\alpha}{;\beta} = g^{\alpha;\bar\gamma}\varphi_{\bar\gamma\beta} = g^{\alpha\bar\gamma}g_{\bar\gamma\beta} = \ind{\delta}{}{\alpha}{\beta}$; thus the first statement follows.

Differentiating the constant function $\abs{\partial\varphi}_\omega^2$ with respect to $\partial_{\gamma}$, we get
\begin{equation}\label{eqn:derivative of length}
\begin{aligned}
0
	=\partial_{\gamma}\ip{\vV,\overline\vV}_\omega
	&=\ip{\nabla_{\partial_{\gamma}}\vV, \overline\vV}_\omega
	+\ip{\vV,\nabla_{\partial_{\gamma}}\overline\vV}_\omega
	\\
	&=\ind{\varphi}{}{\alpha}{;\gamma}\varphi_\alpha
	+\varphi^\alpha\varphi_{\alpha;\gamma}
	=\varphi_\gamma
	+\varphi^\alpha\varphi_{\alpha;\gamma}
	=\varphi_\gamma
	+\varphi^\alpha\varphi_{\gamma;\alpha}
	\;,
    \end{aligned}
\end{equation}
since  $\nabla$ is torsion-free so $\varphi_{\alpha;\gamma}=\varphi_{\gamma;\alpha}$. This implies 
\begin{equation*}
-\varphi^{\bar\beta} = -\varphi_{\gamma}g^{\gamma\bar\beta} = \varphi^\alpha\varphi_{\gamma;\alpha}g^{\gamma\bar\beta} = \varphi^\alpha\ind{\varphi}{}{\bar\beta}{;\alpha} \;.
\end{equation*}
Therefore we get
\begin{equation*}
\nabla_{\vV}\overline\vV 
	= \nabla_{\varphi^\alpha\partial_\alpha}\varphi^{\bar\beta}\partial_{\bar\beta}
	= \varphi^\alpha\nabla_{\partial_\alpha}\varphi^{\bar\beta}\partial_{\bar\beta}
	= \varphi^\alpha\ind{\varphi}{}{\bar\beta}{;\alpha}\partial_{\bar\beta}
	= -\varphi^{\bar\beta}\partial_{\bar\beta}
	=-\overline\vV.
\end{equation*}
This completes the proof.
\end{proof}

\bigskip

\noindent \textit{Proof of Theorem~\ref{thm:integral curve of gradient vector field}}.
Under the assumption $\abs{\partial\varphi}_\omega^2\equiv\cC$, the gradient vector field $\vV$ has nonzero constant length $\cC$, so it is nowhere vanishing on its domain and satisfies  
\begin{equation}\label{eqn:covariant derivative of V}
\nabla_{\vV}\vV=\vV\;, \quad \nabla_{\overline\vV}\vV=-\vV
\end{equation}
by Proposition~\ref{prop:on gradient vector field}. Now we have the real distribution generated by $\vV$, i.e. spanned by $\vV+\overline\vV$ and $\im\paren{\vV-\overline\vV}$, which is of constant rank $2$ and  invariant under the complex structure of $X$. This is involutive since  
\begin{equation}\nonumber
\begin{aligned}
\lb{\vV+\overline\vV,\im\paren{\vV-\overline\vV}}
&=-2\im\lb{\vV,\overline\vV} 
\\
&= -2\im\paren{\nabla_\vV\overline\vV-\nabla_{\overline\vV}\vV} = -2\im\paren{\vV-\overline\vV} \;.
\end{aligned}
\end{equation}
Therefore, for each point $p$ at which $\varphi$ is defined, there is a holomorphic curve $u:U\subset\CC\to X$ passing through $p$ such that $\vV$ is tangent to $S=u(U)$. Equations in \eqref{eqn:covariant derivative of V} imply that the K\"ahler connection of $(U,u^*\omega)$ coincides with $u^*\nabla$; thus $u$ is a totally geodesic curve. Since 
\begin{equation}\nonumber
\begin{aligned}
R(\vV,\overline\vV)\vV
	&=\nabla_\vV\nabla_{\overline\vV}\vV
	-\nabla_{\overline\vV}\nabla_\vV\vV
	-\nabla_{\lb{\vV,\overline\vV}}\vV
	\\
	&=\nabla_\vV(-\vV)
	-\nabla_{\overline\vV}\vV
	-\paren{\nabla_\vV\vV-\nabla_{\overline\vV}\vV}
	=-2\vV \;,
    \end{aligned}
    \end{equation}
the Gaussian curvature $\kappa$ of the induced metric $u^*\omega$ is obtained by
\begin{equation*}
\kappa
	= \frac{\ip{R(\vV,\overline\vV)\vV,\overline\vV}_\omega}{\ip{\vV,\overline\vV}_\omega^2}
	= \frac{-2\ip{\vV,\overline\vV}_\omega}{\ip{\vV,\overline\vV}_\omega^2}
	\equiv \frac{-2}{\cC} \;.
\end{equation*}
This completes the proof. \qed


\subsection{A remark on a K\"ahler submanifold and its induced potential}\label{subsec:submanifold}
Let $\iota:Y\hookrightarrow X$ be a complex submanifold and consider the induced metric $\omega_Y=\iota^*\omega$ with the potential $\varphi_Y=\iota^*\varphi$. The restriction of $\vV$ to $Y$ has the orthogonal decomposition with respect to $\iota:Y\hookrightarrow X$
\begin{equation}\label{eqn:orthogonal decomposition}
\vV = \vV^\top+\vV^\bot
\end{equation}
where $\vV^\top$ is the orthogonal projection of $\vV$ to $T^{(1,0)}Y$ and $\vV^\bot$ is the orthogonal complement of $\vV^\top$ in $T^{(1,0)}X$. 
Then, the tangent vector field $\vV^\top$ of $Y$ is indeed the gradient vector field of the potential function $\varphi_Y$ with respect to  $\omega_Y$, i.e. $\partial\varphi_Y(v)=\ip{v,\overline{\vV^\top}}_{\omega_Y}$ for any $v\in T^{(1,0)}Y$.
This observation can be derived using the identity $\partial\varphi(v)=\ip{v,\overline\vV}_\omega$ of \eqref{eqn:def of gradient} and
\begin{equation}\nonumber
\begin{aligned}
\partial\varphi_Y(v) 
	&= v(\varphi_Y)
	=  v(\varphi )
	= \partial\varphi(v)
	\\
	&=\ip{v,\overline\vV}_\omega 
	= \ip{v,\overline{\vV^\top}+\overline{\vV^\bot}}_\omega 
	= \ip{v,\overline{\vV^\top}}_\omega 
	=\ip{v,\overline{\vV^\top}}_{\omega_Y}
\;.
\end{aligned}
\end{equation}
Therefore we have
\begin{equation}\label{eqn:gradient length of submanifold}
\abs{\partial\varphi_Y}_{\omega_Y}=\abs{\vV^\top}_\omega \;.
\end{equation}


\section{The constant gradient length for the polydisc}\label{sec:polydisc}
In this section, we will prove Theorem~\ref{thm:main1} for the polydisc
\begin{equation*}
\Delta^n=\Delta\times\cdots\times\Delta = \set{(z^1,\ldots,z^n)\in\CC^n:\abs{z^\alpha}<1\text{ for $\alpha=1,\ldots,n$}}
\end{equation*}
whose complete \KE metric with Ricci curvature $-\cK$ is given by
\begin{equation}\label{eqn:KE metric of polydisc}
\omega_{\Delta^n}
	= \im\sum_{\alpha=1}^n \frac{2}{\cK\paren{1-\abs{z^\alpha}^2}^2} dz^\alpha\wedge dz^{\bar\alpha}
	\;.
\end{equation}
For any choice of $\theta_\alpha\in\RR$, a potential of the form
\begin{equation*}
\varphi = \frac{1}{\cK}\log\prod_\alpha\frac{\abs{1-e^{\im\theta_\alpha}z^\alpha}^4}{\paren{1-\abs{z^\alpha}^2}^2}
\end{equation*}
has constant gradient length  $2n/\cK$. This value coincides with $\cL^2_{\Delta^n}/\cK$ in Theorem~\ref{thm:main KE} since $\cL_{\Delta^n}^2 = 2n$. In this section we will show that
\begin{proposition}\label{prop:polydisc uniqueness}
If there is a local potential $\varphi$ of $\omega_{\Delta^n}$ in \eqref{eqn:KE metric of polydisc} with constant gradient length, then the length should be
\begin{equation*}
\abs{\partial\varphi}_{\omega_{\Delta^n}}^2\equiv \frac{2n}{\cK} \;.
\end{equation*}
\end{proposition}

Let us consider the unit ball $\UB^n=\set{z\in\CC^n:\norm{z}<1}$, the irreducible bounded symmetric domain of rank $1$. If $\omega$ is the complete \KE metric with Ricci curvature $-\cK$, then every totally geodesic holomorphic disc (or curve) has constant Gaussian curvature $\kappa=-2\cK/(n+1)$. Thus a constant gradient length of a potential function should be $(n+1)/\cK$ by Theorem~\ref{thm:integral curve of gradient vector field}. But in a bounded symmetric domain with rank $\geq 2$, there are totally geodesic holomorphic discs with different Gaussian curvature. We will show that the gradient vector field $\vV$ in Theorem~\ref{thm:integral curve of gradient vector field} should be tangent to a totally geodesic disc with full rank using Lemma~\ref{lem:rank}. This lemma is also mainly used to show the uniqueness of Theorem~\ref{thm:main KE}.


\subsection{A totally geodesic holomorphic disc in $\Delta^n$}\label{subsec:totall geodesic disc}
As shown in Theorem~\ref{thm:integral curve of gradient vector field}, a potential with constant gradient length gives a local foliation of totally geodesic holomorphic curves whose induced metric has constant Gaussian curvature. A proper totally geodesic holomorphic disc $u:\Delta\to\Delta^n$  is of the form
\begin{equation*}
u(\zeta)=(u^1(\zeta),\ldots,u^n(\zeta))
\end{equation*}
where $u^\alpha:\Delta\to\Delta$ is constant or biholomorphic. We will show that even a local totally geodesic holomorphic curve extends to the same form.

\begin{proposition}\label{prop:totally geodesic disc}
Let $U$ be an open neighborhood of $0$ in $\Delta$ and let $u$ be a holomorphic curve
\begin{equation*}
u=(u^1,\ldots,u^n): U\to\Delta^n \;.
\end{equation*}
Suppose that $u$ is a totally geodesic isometric embedding with respect to the Poincar\'e metric 
\begin{equation*}
\omega_{\Delta}
	=\im  \frac{2}{\cK_u\paren{1-\abs{\zeta}^2}^2} d\zeta\wedge d\bar\zeta
\end{equation*}
with Gaussian curvature $-\cK_u$ and the \KE metric $\omega_{\Delta^n}$ as in \eqref{eqn:KE metric of polydisc}. Then 
\begin{enumerate}
\item\label{enum: totally geodesic disc1} if $(\pdl{u^\alpha}{\zeta})(0)=0$, then $u^\alpha$ is constant;
\item\label{enum: totally geodesic disc2} if $(\pdl{u^\alpha}{\zeta})(0)\neq0$, then  $u^\alpha$ extends to an automorphism $u^\alpha:\Delta\to\Delta$;
\item\label{enum: totally geodesic disc3}  for  the cardinal  number $k$  of the set $\set{\alpha=1,\ldots,n: (\pdl{u^\alpha}{\zeta})(0)\neq0}$, 
\begin{equation*}
\cK_u= \frac{\cK}{k} \;.
\end{equation*}
\end{enumerate}
\end{proposition}

\begin{proof}
Let $D$ and $\nabla$  be the K\"ahler connections of $\omega_\Delta$ and $\omega_{\Delta^n}$.
Since $u:\Delta\rightarrow\Delta^n$ is a totally geodesic embedding, the pull-back connection $u^*\nabla$ coincides with $D$. 
More precisely, we have
\begin{equation*}
	\nabla_{u_*V} u_*W = u_*(D_VW)
\end{equation*}
for any tangent vector fields $V,W$ of $\Delta$. Two K\"ahler connections $D$ and $\nabla$ can be computed as follows.
\begin{align*}
D\pd{}{\zeta} 
	= \frac{2\bar\zeta}{1-\abs{\zeta}^2}d\zeta\otimes\pd{}{\zeta}
	\;,
	\quad
	\nabla\pd{}{z^\alpha}
	= 
	\frac{2 z^{\bar\alpha}}{1-\abs{z^\alpha}^2}dz^\alpha\otimes\pd{}{z^\alpha}.
\end{align*}
Hence it follows that
\begin{equation*}
u_*\paren{D_{\pd{}{\zeta}}{\pd{}{\zeta}}}
	=
	\frac{2\bar u}{1-\abs{\zeta}^2}u_*\paren{\pd{}{\zeta}}
	=
	\frac{2\bar\zeta}{1-\abs{\zeta}^2}u^\alpha_\zeta\pd{}{z^\alpha},
\end{equation*}
and
\begin{align*}
	\nabla_{u_*\paren{\pd{}{\zeta}}}u_*\paren{\pd{}{\zeta}}
	&=
	\nabla_{u_*\paren{\pd{}{\zeta}}}u^\alpha_\zeta\pd{}{z^\alpha}
	=
	\paren{u_*\paren{\pd{}{\zeta}}u^\alpha_\zeta}\pd{}{z^\alpha}
	+
	u^\alpha_\zeta u^\gamma_\zeta\nabla_{\pd{}{z^\gamma}}\pd{}{z^\alpha}
	\\
	&=
	\paren{
		u^\alpha_{\zeta\zeta}
		+
		\paren{u^\alpha_\zeta}^2
		\frac{2u^{\bar\alpha}}{1-\abs{u^\alpha}^2}
	}
	\pd{}{z^\alpha}
\end{align*}
where $u^\alpha_\zeta=\pdl{u^\alpha}{\zeta}$ and $u^\alpha_{\zeta\zeta}=\pdl{u^\alpha_\zeta}{\zeta}=\paren{u_*(\pdl{}{\zeta})}u^\alpha_\zeta$. Therefore, we have
\begin{equation*}
	u^\alpha_{\zeta\zeta}
	+(u^\alpha_\zeta)^2
		\frac{2 u^{\bar\alpha}}{1-\abs{u^\alpha}^2}
	=u^\alpha_\zeta
		\frac{2\bar\zeta}{1-\abs{\zeta}^2}
\end{equation*}
for each $\alpha$. Since $u^\alpha_\zeta$ and $u^\alpha_{\zeta\zeta}$ are holomorphic, the differentiation with respect to $\pdl{}{\bar\zeta}$ gives
\begin{equation}\label{eqn:for totally geodesic curve}
0
	=u^\alpha_\zeta\paren{\frac{2}{\paren{1-\abs{\zeta}^2}^2}
	-\abs{u^\alpha_\zeta}^2\frac{2}{\paren{1-\abs{u^\alpha}^2}^2}}
	\;.
\end{equation}

Suppose that $u^\alpha_\zeta(0)=0$ for some $\alpha$.  If  $u^\alpha_\zeta$ is not identically vanishing, we may assume that $u^\alpha_\zeta$ is nowhere vanishing on $U\setminus\set{0}$ (after by shrinking $U$, if necessary) by the identity theorem of holomorphic functions. The identity~\eqref{eqn:for totally geodesic curve} implies that
\begin{equation}\label{eqn:Schwarz-Pick identity}
\abs{u^\alpha_\zeta}^2
	=\paren{\frac{1-\abs{u^\alpha}^2}{1-\abs{\zeta}^2}}^2
\end{equation}
holds on $U\setminus\set{0}$. As $\zeta$ tends to $0$, it follows that $|u^\alpha_\zeta(0)|=1-\abs{u^\alpha(0)}^2$. Since $u^\alpha$ is a mapping from $U$ to $\Delta$, it follows that $1-\abs{u^\alpha}^2>0$, so $u^\alpha_\zeta(0)\neq 0$. This is a contradiction. Thus $u^\alpha_\zeta(0)=0$  implies that $u^\alpha_\zeta\equiv0$, thus Statement (\ref{enum: totally geodesic disc1}) follows.

\medskip

Suppose that $u^\alpha_\zeta(0)\neq 0$ for some $\alpha$, then we may assume that $u^\alpha_\zeta$ is nowhere vanishing on $U$ by shrinking $U$ if necessary. Then \eqref{eqn:for totally geodesic curve} implies that \eqref{eqn:Schwarz-Pick identity} holds on $U$. Therefore the function $u^\alpha:U\subset\Delta\to\Delta$ preserves the Poincar\'e metric: $(u^\alpha)^*\omega_\Delta=\omega_\Delta$. Using arguments in the Schwarz-Pick Lemma, one can easily obtain Statement (\ref{enum: totally geodesic disc2}).

\medskip

Now we have
\begin{equation}\label{eqn:pulling-back of metric}
    \begin{aligned}
u^*\omega_{\Delta^n}
	&=\im \sum_{\alpha=1}^n \frac{2\abs{u^\alpha_\zeta}^2}{\cK\paren{1-\abs{u^\alpha}^2}^2} d\zeta\wedge d\bar\zeta
	\\
	&= \im \sum_{\afrac{\alpha=1,\ldots,n}{u^\alpha_\zeta(0)\neq 0}} \frac{2}{\cK\paren{1-\abs{\zeta}^2}^2} d\zeta\wedge d\bar\zeta
	= \im \frac{2k}{\cK\paren{1-\abs{\zeta}^2}^2} d\zeta\wedge d\bar\zeta
    \end{aligned}
\end{equation}
by \eqref{eqn:Schwarz-Pick identity}. Since $u$ is an isometric embedding, this is the same as $\omega_\Delta$; thus we get $\cK_u=\cK/k$.
\end{proof}


\subsection{Proof of Proposition~\ref{prop:polydisc uniqueness}} 
Let $u:\Delta\to\Delta^n$ be a totally geodesic isometrically embedded holomorphic disc. The rank of $u$ is defined to be the uniquely determined cardinal number $k=1,\ldots,n$  in Statement~\eqref{enum: totally geodesic disc3} of Proposition~\ref{prop:totally geodesic disc}.

First, we introduce a key lemma to show Proposition~\ref{prop:polydisc uniqueness} and also Theorem~\ref{thm:main KE}.

\begin{lemma}\label{lem:rank}
Let $u:\Delta\to\Delta^n$ be a totally geodesic isometrically embedded holomorphic disc and let $S=u(\Delta)$. Suppose that there is a local potential $\varphi$ of $\omega_{\Delta^n}$ in \eqref{eqn:KE metric of polydisc} with
\begin{equation*}
\abs{\partial\varphi}_{\omega_{\Delta^n}}^2\leq\cC 
\end{equation*}
for some $\cC>0$. If
\begin{enumerate}
\item $\vV=\grad\varphi$ is tangent to $S$, that is, $\vV(p)\in T_p^{(1,0)}S$ for each $p\in S$,
\item $\abs{\partial\varphi}_{\omega_{\Delta^n}}^2\equiv\cC$ on $S$,
\end{enumerate}
then $u$ is of full rank $n$.
\end{lemma}

\begin{proof}
Suppose to the contrary that the rank of $u$ is assumed to be $k<n$.
Taking into account of the automorphisms of $\Delta^n$, $u$ can be written as
\begin{equation*}
u(\zeta)=\{(\underbrace{\zeta,\ldots,\zeta}_k,0,\ldots,0)\;.
\end{equation*}
On the other hand, we have another totally geodesic holomorphic disc $\hat u:\Delta\to\Delta^n$ defined by
\begin{equation*}
\hat u(\xi)=(\underbrace{0,\ldots,0}_k,\xi,\ldots,\xi)\;.
\end{equation*}
Consider the embedding of the bidisc,
\begin{equation*}
v:\Delta^2\hookrightarrow\Delta^n \text{ given by }\iota(\zeta,\xi)=u(\zeta)+\hat u(\xi)=(\zeta,\ldots,\zeta,\xi,\ldots,\xi)
\;,
\end{equation*}
and the induced K\"ahler metric $v^*\omega_{\Delta^n}$ on $\Delta^2$ of the form
\begin{equation*}
v^*\omega_{\Delta^n}
	= \im \lambda d\zeta\wedge d\bar\zeta 
	+\im \mu d\xi\wedge d\bar\xi
\end{equation*}
where 
\begin{equation*}
\lambda=\frac{2k}{\cK\paren{1-\abs{\zeta}^2}^2} \quad\text{and}\quad
\mu=\frac{2(n-k)}{\cK\paren{1-\abs{\xi}^2}^2}
\;.
\end{equation*}
From now on, $u(\Delta)$ and $v(\Delta^2)$ will be denoted by $S$ and $D$.
For simplicity, and to eliminate ambiguity, $\Delta$ and $\Delta^2$ will be identified with (and even denoted as) $S$ and $D$ via the isometric embedding $u$ and $v$.
Furthermore, $\zeta\in\Delta$ and $(\zeta,\xi)\in\Delta^2$ will be used as coordinates of $S$ and $D$, respectively.
So, we have a chain of embeddings
\begin{equation*}
S\underset{u}{\hookrightarrow}
D\underset{v}{\hookrightarrow} \Delta^n \;,
\end{equation*}
where $S$ is also considered as $\set{(\zeta,0):\zeta\in\Delta}$ in $D$.

\medskip

Considering the orthogonal decomposition with respect to $v:\Delta^2\rightarrow\Delta^n$
\begin{equation*}
\vV=\vV^\top+\vV^\bot
\end{equation*}
as in \eqref{eqn:orthogonal decomposition} where $\vV^\top$ is the orthogonal projection to $T^{(1,0)}D$, we have $\vV^\top=\vV$ on $S$ since $S\subset D$ and $\vV$ is tangent to $S$ by the assumption. By \eqref{eqn:gradient length of submanifold}, the potential function $\varphi_D=v^*\varphi$ of $\theta$ satisfies
\begin{equation*}
	\abs{\partial\varphi_D}_{\omega_D}^2
	=
	\abs{\vV^\top}_{\omega_D}^2
	=
	\abs{\vV^\top}_\omega^2
	\leq
	\abs{\vV}_\omega^2\leq\cC
\end{equation*}
and $\abs{\partial\varphi_D}_{\omega_D}^2\equiv\cC$ on $S$, where $\omega_D:=v^*\omega_{\Delta^n}$.
A similar argument can be applied to the embedding $S\hookrightarrow \Delta^n$, which yields that the potential $\varphi_S=u^*\varphi$ (a function of one variable $\zeta$) of the Poincar\'e metric $\im\lambda d\zeta\wedge d\bar\zeta$ has constant gradient length:
\begin{equation*}
\spd{\varphi_S}{\zeta}{\bar\zeta}=\lambda \;, \quad \pd{\varphi_S}{\zeta}\pd{\varphi_S}{\bar\zeta}\frac{1}{\lambda}\equiv\cC.
\end{equation*}
Since 
\begin{equation*}
\eta(\xi)=-\frac{2(n-k)}{\cK}\log(1-\abs{\xi}^2)
\end{equation*}
is a potential of the metric $\im\mu d\xi\wedge d\bar\xi$, the function $\varphi_S(\zeta)+\eta(\xi)$ is a potential of $\omega_D$. Therefore we can write $\varphi_D$ as
\begin{equation*}
\varphi_D(\zeta,\xi)=\psi(\zeta)+\eta(\xi)+f(\zeta,\xi)+\overline{f(\zeta,\xi)}
\end{equation*}
for some local holomorphic function $f$. This implies that
\begin{equation*}
\pd{}{\bar\xi}\pd{\varphi_D}{\zeta}\equiv0,
\end{equation*}
and as a result,  for a fixed $\zeta_0$, the function
\begin{equation}\label{eqn:first of gradient length}
\xi\mapsto \abs{\pd{\varphi_D}{\zeta}(\zeta_0,\xi)}^2 
\end{equation}
is subharmonic.

On the other hand, the gradient vector field of $\varphi_D$ is of the form
\begin{equation*}
\vV^\top=\grad(\varphi_D)=\frac{1}{\lambda}\pd{\varphi_D}{\bar\zeta}\pd{}{\zeta}+\frac{1}{\mu}\pd{\varphi_D}{\bar\xi}\pd{}{\xi} 
\end{equation*}
and $\vV^\top(\zeta,0)=\vV^\top(u(\zeta))$ is in the tangent direction of $S$, i.e. the direction of $\pdl{}{\zeta}$. Therefore we have 
\begin{equation}\label{eqn:holomorphicity of varphi}
\frac{\partial\varphi_D}{\partial \bar\xi}(\zeta,0)
=\frac{\partial\varphi_D}{\partial \xi}(\zeta,0)=0
\;.
\end{equation}

We now consider the length of $\vV^\top$,
\begin{equation*}
\abs{\vV^\top}_{\omega_D}^2 (\zeta_0,\xi)
	=\frac{1}{\lambda(\zeta_0)}\abs{\pd{\varphi_D}{\zeta}(\zeta_0,\xi)}^2
	+\frac{1}{\mu(\xi)}\abs{\pd{\varphi_D}{\xi}(\zeta_0,\xi)}^2
	\;.
\end{equation*}
It  does not exceed $\cC$ and attains its maximum $\cC$ at $(\zeta_0,0)=u(\zeta_0)$ by our assumption.
Together with~\eqref{eqn:holomorphicity of varphi}, this shows that the subharmnic function \eqref{eqn:first of gradient length} achieves its maximum at $(\zeta_0,0)$, which implies that it vanishes identically by the maximum principle.
Hence, the first term of $\abs{\vV^\top}_{\omega_D}^2$ is constant $\cC$; especially the second term vanishes identically for any choice of $\zeta_0$, namely $\pdl{\varphi_D}{\xi}=0$. 
Therefore, we have 
\begin{equation*}
\pd{}{\bar\xi}\pd{\varphi_D}{\xi}\equiv0,
\end{equation*}
but
\begin{equation*}
\quad \pd{}{\bar\xi}\pd{\varphi_D}{\xi}=\spd{\eta}{\bar\xi}{\xi}=\mu>0 \;.
\end{equation*}
This contradiction arises from the assumption $\rank u<n$.
\end{proof}


\begin{proof}[Proof of Proposition~\ref{prop:polydisc uniqueness}]
Let $\varphi:V\subset\Delta^n\to\RR$ be a local potential function of the \KE metric $\omega_{\Delta^n}$ as in \eqref{eqn:KE metric of polydisc} with constant gradient length $\cC$ and let $\vV=\varphi^\alpha\partial_\alpha$ be the gradient vector field of $\varphi$. For a point $p\in V$, there is a holomorphic  integral curve $u:U\subset\Delta\to V$ of $\vV$ passing through $p$ which is  a totally geodesic isometric embedding with respect to the Poincar\'e metric of Gaussian curvature $-\cK_0=-2/\cC$ by Theorem~\ref{thm:integral curve of gradient vector field}. By Proposition~\ref{prop:totally geodesic disc}, $u$ extends globally on $\Delta$ as a totally geodesic holomorphic disc. By applying Lemma~\ref{lem:rank} to $\varphi$ and $S=u(\Delta)$, we obtain $\rank u=n$, so 
\begin{equation*}
\cK_0=\frac{\cK}{n}
\end{equation*}
This implies that $\cC=2n/\cK$ by Proposition~\ref{prop:totally geodesic disc}. 
\end{proof}

\begin{remark}\label{rmk:reducible}
All analogous results in this section hold also for the metric
\begin{equation*}
\tilde\omega
	= \im\sum_{\alpha=1}^n \frac{2}{\cK_\alpha\paren{1-\abs{z^\alpha}^2}^2} dz^\alpha\wedge dz^{\bar\alpha}
\end{equation*}
of $\Delta^n$ for any positive constants $\cK_\alpha$'s. Especially, 
\begin{enumerate}
\item Proposition~\ref{prop:totally geodesic disc} holds by substituting
\begin{equation*}
\cK_u = \paren{\sum_{\alpha=1}^n \frac{1}{\cK_\alpha}}^{-1} 
\end{equation*}
in case of full rank $n$ as derived from the computation of $u^*\tilde\omega$ following \eqref{eqn:pulling-back of metric},
\item then $\cK_0$  in the proof of Proposition~\ref{prop:polydisc uniqueness} coincides with above $\cK_u$,
\item therefore if there is a potential $\varphi$ of $\tilde\omega$ with constant gradient length, then 
\begin{equation*}
\abs{\partial\varphi}_{\tilde\omega}^2\equiv \frac{2}{\cK_u}=\sum_{\alpha=1}^n \frac{2}{\cK_\alpha} \;.
\end{equation*}
\end{enumerate}
\end{remark}


\section{Proof of Theorem~\ref{thm:main KE}}\label{sec:proof}
In this section, we will prove Theorem~\ref{thm:main KE}. With the help of Theorem~\ref{thm:integral curve of gradient vector field} and Proposition~\ref{prop:ball case}, the problems are reduced the curvature computation of a totally geodesic holomorphic disc of full rank in a maximal polydisc. Thus we will first investigate the Gaussian curvature of such disc.

\subsection{The polydisc theorem and the curvature of a totally geodesic holomorphic disc of full rank in a maximal polydisc}
Let $\Omega$ be a bounded symmetric domain of rank $r$ and $\omega$ be the complete \KE metric with Ricci curvature $-\cK$. 

By Wolf's polydisc theorem in \cite{Wolf1972}, for each tangent vector $v$ of $\Omega$, there is a totally geodesic complex submanifold $\Delta^r\hookrightarrow\Omega$ of $(\Omega,\omega)$ such that  $v$ is tangent to $\Delta^r$. We call this $\Delta^r$ a \emph{maximal polydisc} in the direction of $v$. Every maximal polydiscs in $\Omega$ are isometric to each other with respect to the induced metric from $\omega$. Especially if $\Omega$ is irreducible, then the induced metric  is isometric to the complete \KE metric of $\Delta^r$ with some Ricci cuvature $-\cK_m$. In order to get the curvature of totally geodesic holomorphic disc in a maximal polydisc, we will study the value of $\cK_m$ in the irreducible case first. 

\subsubsection{In case of irreducible $\Omega$} If $\Omega$ is irreducible, then 
the metric is given by
\begin{equation*}
\omega = \frac{1}{\cK}dd^c \log K_\Omega(z,z) 
\end{equation*}
where $K_\Omega$ is the Bergman kernel function of $\Omega$. If $\Omega$ is the Cartan/Harish-Chandra embedding, then $K_\Omega$ is given by $K_\Omega = c N_\Omega^{-c_\Omega}$ as we mentioned in the introduction. Then for the maximal polydisc of the form $\Delta^r=\set{(z^1,\ldots,z^r,0,\ldots,0)\in\Omega}$ and its Bergman kernel $K_{\Delta^r}$, we have 
\begin{equation*}
K_\Omega(z,z) = d_1 K_{\Delta^r}(z,z)^{c_\Omega/2} \quad\text{on $\Delta^r$}
\end{equation*}
for some suitable constant $d_1$. Hence, the induced metric of $\Delta^r$ which is given by the potential $\paren{c_\Omega/2\cK}\log K_{\Delta^r}$ is the complete \KE metric with Ricci curvature 
\begin{equation*}
-\cK_m=-\frac{2\cK}{c_\Omega}
\end{equation*}
and a totally geodesic holomorphic disc of full rank $r$ in $\Delta^r$  has constant Gaussian curvature 
\begin{equation*}
\kappa=-\frac{\cK_m}{r}=-\frac{2\cK}{rc_\Omega}=-\frac{2\cK}{\cL_\Omega^2} 
\end{equation*}
by Proposition~\ref{prop:totally geodesic disc}.

\subsubsection{In case of reducible $\Omega$}\label{subsubsec:reducible}
If $\Omega=\Omega_1\times\cdots\times\Omega_s$ with irreducible factors $\Omega_j$, $j=1,\ldots,s$, a maximal polydisc in $\Omega$ is given by
\begin{equation*}
\Delta^r = \Delta^{r_1}\times\cdots\times\Delta^{r_s}
\end{equation*}
where each $\Delta^{r_j}$ is a maximal polydisc of $\Omega_j$. Since $\omega$ is the product metric of the complete \KE metric of $\Omega_j$ with Ricci curvature $-\cK$, the induced metric on $\Delta^r$ is the product metric of the complete \KE metric of $\Delta^{r_j}$ with Ricci curvature 
\begin{equation*}
-\cK_j = -\frac{2\cK}{c_{\Omega_j}} 
\end{equation*}
where $c_{\Omega_j}$ is the genus of $\Omega_j$. By Remark~\ref{rmk:reducible}, we can conclude that a totally geodesic holomorphic disc of full rank $r=r_1+\cdots+r_s$ in a maximal polydisc of $\Omega$ has constant Gaussian curvature
\begin{equation}\label{eqn:curvature of maximal rank}
\kappa= -\paren{\sum_{j=1}^s \frac{r_j}{\cK_j}}^{-1} = -\paren{\sum_{j=1}^s\frac{r_j c_{\Omega_j}}{2\cK}}^{-1} = -\frac{2\cK}{\cL_\Omega^2} \;.
\end{equation}


\subsection{The rigidity of constant gradient length}\label{subsec:proof1}
We will prove Statement (\ref{thm:main KE1}) of Theorem~\ref{thm:main KE}. Suppose that $\varphi:V\subset\Omega\to\RR$ be a potential of $\omega$ with constant gradient length, $\abs{\partial\varphi}_\omega^2\equiv\cC$. Let $p$ be a point in $V$. 

By Theorem~\ref{thm:integral curve of gradient vector field}, we have an integral holomorphic curve $S$ of $\vV=\grad(\varphi)$ passing through $p$ such that $S$ is totally geodesic and $\omega|_S$ is a hermitian metric with constant Gaussian curvature $-2/\cC$. 

Let us consider a maximal polydisc $M$ passing through $p$ in the direction of $\vV(p)$. Both $M$ and $S$ are totally geodesic and $T_pS\subset T_pM$; thus $S$ is also a totally geodesic holomorphic curve of $M$, that is, $S\subset M$. Regarding $\vV$ as a vector field on $M$ (i.e. $\vV=\vV|_M$), we consider the orthogonal decomposition
\begin{equation*}
\vV = \vV^\top+\vV^\bot
\end{equation*}
where $\vV^\top$ the orthogonal projection to $T^{(1,0)}M$ and $\vV^\bot$ is the orthogonal complement of $\vV^\top$ in $T^{(1,0)}\Omega$. 
As shown in Section~\ref{subsec:submanifold}, the tangent vector field $\vV^\top$ of $M$ is the gradient vector field of the potential function $\varphi_M=\varphi|_M$ with respect to  the induced metric $\omega_M=\omega|_M$. Thus we have
\begin{equation*}
\abs{\partial\varphi_M}_{\omega_M}^2
	=\abs{\vV^\top}_\omega^2
	\leq\abs{\vV}_\omega^2=\cC.
\end{equation*}
Since $\vV$ is tangent to $S\subset M$, we have $\vV=\vV^\top$ on $S$, so that
\begin{equation*}
\abs{\vV^\top}_\omega^2
	=\abs{\vV}_\omega^2=\cC \quad\text{on $S$.}
\end{equation*}
Applying Lemma~\ref{lem:rank} and Remark~\ref{rmk:reducible} to $(M,\omega_M)$, the totally geodesic disc $S$ is of full rank $r$. Therefore the Gaussian curvature $\kappa$ of $S$ is the constant given by \eqref{eqn:curvature of maximal rank}. Thus  we have
\begin{equation*}
\cC= -\frac{2}{\kappa} =  \sum_{j=1}^s\frac{r_j c_{\Omega_j}}{\cK} = \frac{\cL_\Omega^2}{\cK}
\end{equation*}
by Theorem~\ref{thm:integral curve of gradient vector field}. This completes the proof. \qed


\subsection{The lower bound of $\norm{\partial\varphi}_{L^\infty}$}\label{subsec:proof2}
We will show Statement~(\ref{thm:main KE2}) of Theorem~\ref{thm:main KE} for the unit ball $\UB^n=\set{z\in\CC^n: \norm{z}<1}$. In this case, the K\"ahler-hyperbolic length $\cL_{\UB^n}$ is given by $\cL_{\UB^n}^2=(n+1)/\cK$.

\begin{proposition}\label{prop:ball case}
Let $\omega$ be the complete \KE metric of $\UB^n$ with Ricci curvature $-\cK$ and $\varphi:\UB^n\to\RR$ be a potential of $\omega$. Then 
\begin{equation*}
\sup\abs{\partial\varphi}_\omega^2 \geq \frac{n+1}{\cK}.
\end{equation*}
\end{proposition}

\begin{proof}
For the Laplace-Beltrami operator $\triangle_\omega$, locally  written by $\triangle_\omega = g^{\alpha\bar\beta}\partial_{\alpha}\partial_{\bar\beta}$ for a function, we have proved 
\begin{equation*}
\triangle_\omega\abs{\partial\varphi}_\omega^2
	=\varphi_{\alpha;\beta}\varphi^{\alpha;\beta}+n-\cK\abs{\partial\varphi}_\omega^2
	\;
\end{equation*}
in \cite{CLS2023}. Here, $\varphi_{\alpha;\beta}$ stands for the covariant derivative of the gradient vector field $\grad(\varphi)$ as in the proof of Proposition~\ref{prop:on gradient vector field}.

If $\abs{\partial\varphi}_\omega^2$ attains its maximum at $p\in\mathbb B^n$, then $d\abs{\partial\varphi}_\omega^2=0$ at $p$. Equation~\eqref{eqn:derivative of length} implies $\varphi_{\alpha;\beta}\varphi^\beta=-\varphi_\alpha$ at $p$. This means that the tangent vector $\varphi^\alpha\partial_\alpha$ at $p$ is an eigenvector of
\begin{equation*}
L: v^\alpha\partial_\alpha \mapsto \ind{\varphi}{}{\alpha}{;\bar\beta}\ind{\varphi}{}{\bar\beta}{;\gamma}v^\gamma\partial_\alpha
\end{equation*}
with the eigenvalue $1$. This operator $L$ is  nonnegative definite and symmetric, so we obtain
\begin{equation*}
\trace L=\varphi_{\alpha;\beta}\varphi^{\alpha;\beta}\geq 1 \quad\text{at $p$.}
\end{equation*}
Simultaneously, $\triangle_\omega\abs{\partial\varphi}_\omega^2\leq 0$ at $p$, thus we have
\begin{equation*}
\cK\abs{\partial\varphi}_\omega^2\geq  \varphi_{\alpha;\beta}\varphi^{\alpha;\beta}+n\geq n+1 \quad\text{at $p$.}
\end{equation*}
This gives the desired inequality.

\medskip

Suppose that $\abs{\partial\varphi}_\omega^2$ is uniformly bounded but does not attain its maximum. Then we can take a sequence $\set{p_j}$ of points in $\UB^n$ such that
\begin{equation*}
\abs{\partial\varphi}_\omega^2(p_j)\to\sup\abs{\partial\varphi}_\omega^2.
\end{equation*}
For the origin $o$, we have a sequence $\set{f_j}$ of automorphisms of $\UB^n$ satisfying $f_j(o)=p_j$. Due to the method of potential rescaling developed in \cite{LKH2021,CL2021,CLS2023}, the sequence of potentials given
\begin{equation*}
\varphi_j=\varphi\circ f_j-(\varphi\circ f_j)(o)
\end{equation*}
has a subsequence that converges to a potential $\tilde\varphi:\UB^n\to\RR$ in a local $C^\infty$ convergence since $\abs{\partial\varphi}_\omega$ is uniformly bounded. When we assume $\varphi_j\to\tilde\varphi$ passing to a subsequence, it follows that
\begin{equation*}
\abs{\partial\tilde\varphi}_\omega^2(p) = \lim_{j\to\infty}\abs{\partial\varphi_j}_\omega^2(p) =\lim_{j\to\infty}\abs{\partial\varphi}_\omega^2(f_j(p)) \leq\sup\abs{\partial\varphi}_\omega^2
\end{equation*}
for any $p\in\UB^n$ so $\abs{\partial\tilde\varphi}_\omega^2$ attains its maximum value $\abs{\partial\tilde\varphi}_\omega^2(o)=\lim_{j\to\infty}\abs{\partial\varphi}_\omega^2(p_j)=\sup\abs{\partial\varphi}_\omega^2$ at $o$. Therefore, we have $\sup\abs{\partial\varphi}_\omega^2\geq (n+1)/\cK$.
\end{proof}

For the unit disc $\Delta$ with the Poincar\'e metric $\omega_\Delta$ with Gaussian curvature $\kappa$, we have a conclusion that 
\begin{equation*}
\sup_\Delta \abs{\partial\varphi}_{\omega_\Delta}^2 \geq \frac{2}{-\kappa}
\end{equation*}
for any potential $\varphi:\Delta\to\RR$ of $\omega_\Delta$. Using this, we complete the proof of Theorem~\ref{thm:main KE}.

\medskip

\begin{proof}[Proof of Statement~(\ref{thm:main KE2}) in Theorem~\ref{thm:main KE}] Let $\Omega$ be a bounded symmetric domain and $\varphi:\Omega\to\RR$ be a global potential of $\omega$.

Let $\Delta\hookrightarrow\Omega$ be a totally geodesic holomorphic disc with full rank $r$ in a maximal polydisc. On $\Delta$, $\tilde\omega=\omega|_{\Delta}$ is the Poincar\'e metric with constant Gaussian curvature $\kappa = -2\cK/\cL_\Omega^2$ as showned in Section~\ref{subsubsec:reducible}. By Proposition~\ref{prop:ball case}, we have
\begin{equation*}
\sup_\Delta \abs{\partial\tilde\varphi}_{\tilde\omega}^2 \geq \frac{2}{-\kappa}= \frac{\cL_\Omega^2}{\cK}
\end{equation*}
for $\tilde\varphi=\varphi|_\Delta$. Since $\abs{\partial\tilde\varphi}_{\tilde\omega}^2\leq \abs{\partial\varphi}_{\omega}^2$ on $\Delta$, we have
\begin{equation*}
\sup_\Omega\abs{\partial\varphi}_\omega^2 \geq \sup_\Delta \abs{\partial\tilde\varphi}_{\tilde\omega}^2 \geq \frac{\cL_\Omega^2}{\cK} \;.
\end{equation*}
This completes the proof.
\end{proof}


\end{document}